\documentclass[reqno,b5paper]{amsart}
\usepackage[latin1]{inputenc}
\usepackage{amssymb,amsmath,latexsym}
\usepackage[varg]{pxfonts}
\usepackage{lineno,hyperref}
\usepackage{graphicx}
\usepackage{mathrsfs}
\usepackage[all]{xy}
\usepackage{amssymb}
\usepackage[active]{srcltx} 
\usepackage[active]{srcltx} 
\usepackage[active]{srcltx} 
\usepackage{amssymb,amsmath,latexsym}
\usepackage[varg]{pxfonts}
\usepackage{lineno,hyperref}
\usepackage{amssymb}
\usepackage[all]{xy}
\usepackage[active]{srcltx} 
\usepackage{mathrsfs}
\usepackage{amssymb}
\usepackage{graphics}
\usepackage{amssymb}
\usepackage{color}
\usepackage{amsmath}
\usepackage{graphicx}
\usepackage{caption}
\usepackage{subcaption}
\usepackage{amssymb,latexsym}
\usepackage{enumerate}

 \newtheorem{theorem}{Theorem}[section]

\newtheorem{corollary}[theorem]{Corollary}

\newtheorem{lemma}[theorem]{Lemma}

\newtheorem{proposition}[theorem]{Proposition}

\begin{document}

\title[Chain transitivity in uniform hyperspaces]{Uniform Chain transitivity and Uniform Chain Mixing properties in uniform hyperspaces }
\author[Pirfalak \and Wu\and  Ahmadi \and Kouhestani]{F. Pirfalak$^1$ \and X. Wu$^2$\and S.A. Ahmadi$^3$ \and N. Kouhestani$^4$}
\newcommand{\acr}{\newline\indent}
\address{\llap{$^{1,3,4}$\,}Department of Mathematics\acr
Faculty of mathematics, statistics and computer science\acr
university of Sistan and Baluchestan\acr
Zahedan\acr
Iran}
\email{sa.ahmadi@math.usb.ac.ir}
\email{kouhestani@math.usb.ac.ir}
\address{\llap{$^{2}$\,}School of Sciences\acr
Southwest Petroleum University\acr
 Chengdu, Sichuan 610500
People's
Republic of China}
\email{wuxinxing5201314@163.com}

\subjclass{Primary XXX, ...; Secondary (optional) YYY, ...}
\keywords{....}
\begin{abstract}
We introduce and study the topological concepts of chain transitivity, mixing and chain mixing property for dynamical systems induced by uniform hyperspaces. These notions generalize the relevant concepts for metric spaces.
\end{abstract}

\maketitle

\section*{Uniform space}
 Let $ X $ be a non-empty set. A uniformity  $ \mathscr{U}$ on the set $ X $ is a subset of the product $ X \times X $ \cite{ahmadi}.\\
\begin{itemize} 
\item[U1)]
for any $E_1,E_2\in\mathscr{U}$, the intersection $E_1\cap
E_2$ is also contained in $\mathscr{U},$ and if $E_1\subset E_2$ and $E_1\in\mathscr{U}$, then
$E_2\in\mathscr{U}$;
\item[U2)]
every set $E\in\mathscr{U}$ contains the diagonal $\Delta_X =
\{(x,x) :  x\in X\}$;
\item[U3)]
if $E\in\mathscr{U}$, then $E^{T} = \{(y,x) : (x,y)\in E\} \in\mathscr{U}$;
\item[U4)]
for any $E\in\mathscr{U}$ there exists $\hat{E}\in\mathscr{U}$ such that $\hat{E}\circ \hat{E}
\subset E$, where $$\hat{E}\circ \hat{E} =\{(x,y) :  \exists z\in
X\textrm { with }(x,z)\in \hat{E}, (z,y)\in \hat{E}\}.$$
\end{itemize}
The set $ X $ with a uniformity $ \mathscr{U} $ on it is called \textit{uniform space} and denoted by $ (X, \mathscr{U}) $.\\
Let $ (X,\mathscr{U}) $ be a uniform space. Then each element of $\mathscr{U}$ is called  \textit{entourage} of $ X $. An entourage $ E $ is called \textit{symmetric} if $ E = E^{T} $. If  $x\in X$ and $ A \subset X $ and $ E \in \mathscr{U} $, then the set $E[x]=\{y\in X: (x,y)\in E\}$ is said to be the \textit{cross-section} of $E$ at a point $x$ and 
$$ E[A] = \{ y \in X : \exists a \in A, \  \textit{such that} \  y \in E[a] \}. $$
Moreover for every $ n \in \Bbb N $ denote
\begin{align*}
E^n:=&E\circ E\circ\dots \circ E~(n\mbox{ times})\\
=&\{(x,y): \exists z_0=x,z_1,\dots , z_n=y;  (z_{i-1},z_i)\in E, \mbox{ for }i=1,2,\dots,  n\}.
\end{align*}
 If $ \tau =\{ A \subset X: \forall a \in A, \exists E \in \mathscr{U}, \  \textit{such that} \  E[a] \subset A \} $,  then $ \tau $ is a topology on $ X $ that is called the \textit{topology inducd} by uniformity $ \mathscr{U} $. From now on, when we say that the uniform space $ (X,\mathscr{U}) $ satisfies a certain topological property, we mean that the topological space $ (X, \tau) $ has the same property\cite{ahmadi}.

 In a uniform space $ (X, \mathscr{U}) $, we define hyperspaces of $ X $ as follows:
 \begin{itemize}
\item[1)]
$ 2^{X} =\{ A \subset X : A \ is \ closed \ and \ non-empty \},$
\item[2)]
$ C(X) = \{ A \in 2^{X} : A \ is \ connected \},$\\
and for every $ n \in \Bbb N $,
\item[3)]
$ C_{n}(X) = \{ A \in 2^{X} : A \ has \ at \ most \ n \ components \},$
\item[4)]
$ F_{n}(X) = \{ A \in 2^{X} : A \ has \ at \ most \ n \ points\}, $\\
and also
\item[5)]
$ F(X) = \bigcup_{n=1}^{\infty} F_{n}(X) - the \ collection \ of \ all \ finite \ subsets \ of \ X .$
\end{itemize}
Let $ (X, \mathscr{U}) $ be a uniform space and $ E \in \mathscr{U} $. If  
\[ 2^{E} = \{ ( A, A^{\prime} ) \in 2^{X} \times 2^{X} : A \subset E[A^{\prime}] , \ A^{\prime} \subset E[A] \} ,\]
then it is easy to prove that the set $ \mathfrak{B} = \{ 2^{E} : E \in \mathscr{U} \} $ is a base for a uniformity on $ 2^{X} $, that denoted by  
$$2^{\mathscr{U}}=\{\mathcal{U}\subset2^X\times 2^X: \mbox{there exists } U\in\mathscr{U}~\mbox{such that }2^U\subset\mathcal{U} \}.$$
 It is known that the topology induced by $ 2^{\mathscr{U}} $ coincides with the \textit{Vietoris topology}. We recall that a Vietoris topology on a set $ X $ is a topology with a base of all the sets of the form  
\[ \mathcal{V}(U_{1}, U_{2}, \dots , U_{k}) = \{ B \in 2^{X}: B \subset \bigcup_{i=1}^{k}U_{i} \  \textit{and} \  B \cap U_{i} \neq \emptyset \  \textit{for} \  i=1,2, \dots , k \}, \] where $ U_{1}, U_{2}, \dots , U_{k} $ are open in  $ (X, \tau_{\mathscr{U}}) $ \cite{ahmadi}. 
 By \cite{MR3528853}, if $ (X, \mathscr{U}) $ is compact and Hausdorff, the $ 2^{X} $ is also compact and Hausdorff. \\
Let $ f: X \longrightarrow X $ be a map. We define $ 2^{f}: 2^{X} \rightarrow 2^{X} $ by $ 2^{f}(A) = f(A) $ for all $ A \in 2^{X} $. We defined the maps  $  C_{n}(f), C(f), f_{n} \ and \  f^{<\omega} $ from $ 2^{X} $ into $ 2^{X} $ by 
\[ C_{n}(f) = 2^{f} \mid_{C_{n}(X)}, \ \ C(f) = 2^{f} \mid_{C(X)},\]
\[ f_{n} = 2^{f} \mid_{F_{n}(X)}, \ \ f^{<\omega} = 2^{f} \mid_{F(X)} .\]
These functions are also called  \textit{induced maps}  by $ f $.
\section*{Some definition of dynamical systems}
Let $ X $ is a uniform space and $ f:X \rightarrow X $ is a continuous map.\\
Let $ D $ be an entuorage of uniform space $ (X, \mathscr{U}) $. 
A finite $ D $-pseudo orbit $ \{ x_{0}, x_{1}, x_{2}, \dots , x_{m} \} $ such that $ \left(  f(x_{i}), x_{i+1} \right)  \in D $ for $ 0 \leq i \leq (m-1) $, is called \textit{$ D $-chain } of length $ m $\cite{MR1687407}.\\
 A subset $ \Lambda $ of $ X $ is \textit{uniform internally  chain transitive} if for every pair of points $ x, y \in \Lambda $ and every $ D \in \mathscr{U} $ there is a $ D $-chian $ \{ x_{0}=x, x_{1}, \dots , x_{m}=y \} \subseteq \Lambda $ between $ x $ and $ y $. In the special case when $ \Lambda = X $, we say that $ f $ (or $ X $ ) is \textit{uniform chain transitive}.\\
If $ a, b, c \in X $ and,
\[\Gamma_{1} = \{ l_{0}^{1}=a, l_{1}^{1}, \dots , l_{k_{1}}^{1} =b \} \]
 and 
\[ \Gamma_{2} = \{ l_{0}^{2}=b, l_{1}^{2}, \dots , l_{k_{2}}^{2} = c \} ,\]
are two $ D- $ chains in $ X $, then 
\[ \Gamma_{1} + \Gamma_{2} = \{ l_{0}^{1}=a, l_{1}^{1}, \dots , l_{k_{1}}^{1} = b = l_{0}^{2}, l_{1}^{2}, \dots , l_{k_{2}}^{2} = c \}  \]
 is \textit{concatenation}  of $ \Gamma_{1} $ with $ \Gamma_{2} $.
  Note that the length of  $ \Gamma_{1} + \Gamma_{2} $ is the sum of lengths of $ \Gamma_{1} $ and $ \Gamma_{2} $. \\
 We denot the product space $ \underbrace{X \times X \times \dots \times X}_{n- times} $ by $ X^{(n)} $ and define  $ f^{(n)} :X^{(n)} \rightarrow X^{(n)}$ by $ f^{(n)} (x_{1}, \dots , x_{n}) = ( f(x_{1}), \dots , f(x_{n}) ) $.
This is not to be confused with $ f^{n} $ which is simply the nth iterate of $ f $.\\
We say that $ f $ is \textit{uniform chain weakly mixing}( or \textit{uniform weakly mixing by chains})if the function $ f^{(2)}: X^{(2)} \longrightarrow X^{(2)} $ is uniform chain transitive.\\
The map $ f $ is also \textit{totally uniform chain transitive} ( or \textit{totally uniform transitive by chains} ) if for every $ n \geq 1 $, the function $ f^{n}: X\longrightarrow X $ is uniform chain transitive.\\
We say that $ f $ is
\begin{itemize}
\item[i.]
 \textit{exact by uniform chains} if for every $ E \in \mathscr{U} $ and every non-empty open subset $ U $ of $ X $, there is a positive integer $ n_{E} \geq 1 $ such that for every $ x \in X $ there exists $ u \in U $ and an $ E $-chain $ \{ u= a_{0}, a_{1}, \ldots , a_{n_{E}}= x\} $ from $ u $ to $ x $ with length exactly $ n_{E} $.\\
Recall that a map $ f: X \longrightarrow X $ is called \textit{exact} if for every non-empty open set $ U \subseteq X $, there exists a positive integer $ m $ such that $ f^{m}(U) = X $, and the map $ f $ is called $ uniform \ weakly \ mixing $ if $ f \times f: X \times X \longrightarrow X \times X $ is uniform transitive.
\item[ii.]
\textit{uniform chain mixing}\cite{WU2019145} ( or \textit{mixing by uniform chains} ) if for every entourage $ E $ there exists a positive integer $ N $ such that for all $ n \geq N $ and for any pair $ x,y \in X $, there exists an $ E- $ chain from $ x $ to $ y $ of length $ n $, $ \{ x_{0}= x, x_{1}, x_{2}, \ldots , x_{n}= y \} $.
\item[iii.]
 \textit{uniform chain recurrent}\cite{RICHESON2008251} if for every $ x \in X $ and every entourage $ E $, there is an $ E- $ chain from $ x $ to itself.
 \end{itemize}
\section{Main resultes}
\begin{proposition} \label{A}
Let $ ( X, \mathscr{U} ) $ be a uniform space and let $ f:X \longrightarrow X $ be a continuous function. If $ C(f) , C_{n}(f) , f_{n} , f^{<\omega}$ or $ 2^{f} $, for any $ n $, is uniform chain transitive, then $ f $ is uniform chain transitive.
\end{proposition}
\begin{proof}
We give the proof for $ 2^{f} $, the proofs for the other functions are similar.\\
Let $ E \in \mathscr{U} $, then $ 2^{E} \in 2^{\mathscr{U}} $, and assume that $ 2^{f} $ is uniform chain transitive. Let $ x,y \in X $. Since $ \lbrace x \rbrace , \lbrace y \rbrace \in 2^{X} $, there is an uniform $ 2^{E}-$ chain in $ 2^{X} $,say  $ \lbrace \lbrace x \rbrace = A_{0},A_{1},A_{2}, \ldots ,A_{k} = \lbrace y \rbrace \rbrace $ with $ k \geq 1 $.\\
Now, since $ \left( 2^{f}\left( \left\lbrace x \right\rbrace \right) , A_{1} \right)  \in 2^{E} $ and $ 2^{f}\left( \left\lbrace  x \right\rbrace \right) \subseteq E[A_{1}] $, thus there is $ a_{1} \in A_{1} $ such that $ \left( f(x),a_{1}\right) \in E $. Since $ \left( 2^{f}(A_{1}),A_{2}\right) \in 2^{E} $, then $ 2^{f}(A_{1}) \subseteq E[A_{2}] $, there is $ a_{2} \in A_{2} $ for which $ \left( f(a_{1}),a_{2}\right) \in E $. Following this process there is $ a_{i+1} \in A_{i+1} $ such that $ \left( f(a_{i}),a_{i+1}\right) \in E $. Therefore, the sequence $ \left\lbrace x=a_{0},a_{1},a_{2}, \ldots ,a_{k}=y \right\rbrace  $ is an uniform $ E $-chain from $ x $ to $ y $ in $ X $.
\end{proof}
\begin{lemma}
Let $ X $ be compact Hausdorff space, let $ f: X \longrightarrow X $ be a continuous function and let $ Y $ be a dense and invariant subset of $ X $. Then $ f $ is uniform chain transitive iff $ f_{\mid Y} $ is uniform chain transitive. 
\end{lemma}
\begin{proof}
Assume that $ f $ is uniform chain transitive, let $ a,b \in Y $ and let $ E \in \mathscr{U} $, let $ D \in \mathscr{U} $ be such that $ D^{2} \subset E $. Since $ f $ is uniformity continuous function, there exist $ W $ and $ W^{\prime} $ in $ \mathscr{U} $ such that $ W \circ W^{\prime} \subset D $, and if $ (s,t) \in W $ then $ \left( f(s),f(t) \right) \in W^{\prime} $ for every $ s,t \in X $. Since $ f $ is uniform chain transitive, there is an $ D $-chain in $ X $ from $ a $ to $ b $, $ \left\lbrace a=z_{0},z_{1},z_{2}, \ldots ,z_{k}=b \right\rbrace  $, with $ k \geq 1 $. Since $ Y $ is dense in $ X $, we conclude that for each $ i \in \left\lbrace 1,2, \ldots , k-1 \right\rbrace  $, there is $ t_{i+1} \in Int_{X} \left( D[z_{i+1}] \right) \cap Y $.\\
Since $ W \circ W^{\prime} \subset D $ and $ D^{2} \subset E $, we have $ \left( f(t_{i}),t_{i+1} \right) \in E $. Therefore, the sequence $ \left\lbrace a,t_{1},t_{2},t_{3}, \ldots , t_{k-1},b \right\rbrace  $ is an uniform $ E- $chain in $ Y $. Now assume that $ f_{\mid Y} $ is uniform chain transitive, let $ x,y \in X $, suppose that $ y^{\prime} \in f^{-1}(y) $ and let $ E \in \mathscr{U} $. Also, let $ D \in \mathscr{U} $ such that if $ (a,b) \in D $, then $ \left( f(a),f(b) \right)  \in E $. Since $ Y $ is dense in $ X $, there are $ z_{0},z_{1} \in Y $ such that $ \left( z_{0} , f(x) \right)  \in E $  and $ (z_{1},y^{\prime}) \in D $. This implies that $ \left( f(z_{1}),y\right) \in E $. By hypothesis, there is an uniform $ E- $chain in $ Y $, $ \left\lbrace z_{0}=a_{0},a_{1},a_{2}, \ldots , a_{k}=z_{1} \right\rbrace  $, with $ k \geqslant 1 $. Hence, the sequence $ \left\lbrace x,z_{0}=a_{0},a_{1},a_{2}, \ldots , a_{k}=z_{1},y \right\rbrace  $ is an uniform $ E $-chain from $ x $ to $ y $ in $ X $.
\end{proof}
Recall that for every $ A \in 2^{X} $ and any $ E \in \mathscr{U} $, there is a finite set $ K \in F(X) \subseteq 2^{X} $ such that $ A \subset E[K] $. This along with last lemma gives $ f^{^{< \omega}} $ is uniform chain transitive, iff $ 2^{f} $ is uniform chain transitive.
\begin{lemma}
Let $ X $ be a compact Hausdorff space, and let $ f: X \longrightarrow X $ be a continuous function. If $ f_{n} $ is uniform chain transitive for any $ n \geq 1 $, then $ f^{< \omega} $ is uniform chain transitive.
\end{lemma}
\begin{proof}
Assume $ f_{n} $ is uniform chain transitive for any $ n \geq 1 $, let $ A, B \in F(X) $ of $ F_{n} $`s, we can find $ n \in \Bbb N $ such that both $ A $ and $ B $ belong to $ F_{n}(X) $. Hence, there exists an uniform chain from $ A $ to $ B $. Thus, by defenition of $ F(X) $, we have an uniform chain from $ A $ to $ B $ in $ F(X) $.Therefore, $ f^{< \omega} $ is uniform chain transitive.
\end{proof}
\begin{lemma}
Let $( X,\mathscr{U}) $ be a compact Hausdorff space and let $ f: X \longrightarrow X $ be a continuous function. If $ 2^{f} $ is uniform chain transitive, then $ f_{n} $ is uniform chain transitive for some $ n \geqslant 2 $.
\end{lemma}
\begin{proof}
Assume that $ 2^{f} $ is uniform chain transitive. Let $ A = \{ x_{0},y_{0} \} $ , $ B = \{ x, y \}  $ be two points in $ F_{2}(X) $,let $ z \in X $, and let $ E \in \mathscr{U} $. Since $ 2^{f} $ is uniform chain transitive, then there is an uniform $ 2^{E} -$ chain $ \Gamma_{1} = \{ A = A_{0}, A_{1}, A_{2}, \dots ,A_{n} =\{ z \} \} $ from $ A $ to $ \{ z \}  $ in $ 2^{X} $. Since $ \left( f(A_{i}), A_{i+1} \right)  \in 2^{E} $, there are points $ x_{i}, y_{i} \in A_{i} $ such that $ \left( f(x_{i}), x_{i+1} \right) \in E $ and $ \left( f(y_{i}), y_{i+1} \right) \in E $, for every $ i \in \{ 0,1,2, \dots , n-1 \}  $, let $ A_{i}^{*} = \left\lbrace x_{i},y_{i} \right\rbrace  $. Then $ \Gamma_{1}^{*} = \left\lbrace A^{*} = A_{0}^{*}, A_{1}^{*}, A_{2}^{*}, \dots , A_{n}^{*} = \left\lbrace z \right\rbrace \right\rbrace  $ is an uniform $ E $-chain from $ A $ to $ \left\lbrace z \right\rbrace  $ in $ F_{2}(X) $. Now, since $ 2^{f} $ is uniform chain transitive, then there is an $ E- $ chain $ \Gamma_{2} = \left\lbrace \left\lbrace z \right\rbrace = B_{0}, B_{1}, B_{2}, \dots , B_{l}=B \right\rbrace  $ from $ \left\lbrace z \right\rbrace  $ to $ B $. Let us rename the points in $ B = \left\lbrace x , y \right\rbrace  $ to $ B = \left\lbrace x_{l},y_{l} \right\rbrace  $. Since $ \left( f(B_{l-1}), B_{l} \right) \in 2^{E} $, then for $ x_{l} \in B_{l} $, there is $ x_{l-1} \in B_{l-1} $ such that $ \left( f(x_{l-1}), x_{l} \right) \in E $, and for $ y_{l} \in B_{l} $, there is $ y_{l-1} \in B_{l-1} $ such that $ \left( f(y_{l-1}), y_{l} \right) \in E $. Continuing this process we obtain the sets $ B_{i}^{*}=\left\lbrace x_{i}, y_{i} \right\rbrace  $ such that $ \left( f(x_{l-1}), x_{l} \right) \in E $ and $ \left(  f(y_{l-1}), y_{l} \right) \in E $ for every $ i \in \left\lbrace 1,2, \dots , l \right\rbrace  $. Thus, $ \Gamma_{2}^{*} = \left\lbrace B_{0}^{*}=\left\lbrace z \right\rbrace , B_{1}^{*}, B_{2}^{*}, \dots , B_{l}^{*} = B \right\rbrace  $ is an $ E $-chain from $ \left\lbrace z \right\rbrace  $ to $ B $ in $ F_{2}(X) $.Therefore, the concatenation of $ \Gamma_{1}^{*} + \Gamma_{2}^{*} $ is an $ E- $chain from $ A $ to $ B $ in $ F_{2}(X) $.
\end{proof}
\begin{lemma}
Let $( X,\mathscr{U}) $ be a compact Hausdorff space and let $ f: X \longrightarrow X $ be a continuous function. If $ f_{n} $ is uniform chain transitive for any $ n \geqslant 1 $, then $ 2^{f} $ is uniform chain transitive.
\end{lemma}
\begin{proof}
Assume $ f_{n} $ is uniform chain transitive for any $ n \geq 1 $. Let the sets $ A, B \in 2^{X} $ and let $ E \in \mathscr{U} $ such that there exists an uniform chain from $ A $ to $ B $. Let $ \tilde{E}$ be a symmetrice entourage such that $ \tilde{E}^{2} \subset E $. Let $ D $ be an $ \tilde{E} $-modulus of cotinuity. Since $ f_{n} $ is uniform chain transitive, there exists $ K \in F_{m} $ and $ L \in F_{n} $ for $ n > m $, such that $ A \subset D[K] $ and $ B \subset D[L] $, since $ F_{m} $ is uniform chain transitive, there is a $ D $-chain $ \{ K=K_{0}, K_{1}, \dots , K_{n}= L\} $ between $ K $ and $ L $. Therefore we have $ ( f_{n}(K), K_{1}) \in 2^{D} $ also $ (f_{n}(K_{n}), L) \in 2^{D} $, on the other hand, since $ A \subset D[K] $ and by continuoity, we have 
\[ f(A) \subset f(D[K]) \subset \tilde{E}[f(K)] \subset \tilde{E}[ D[K_{1}] ] \subset \tilde{E}^{2}[K_{1}] \subset E[K_{1}]. \]
Similarly, since $ B \subset D[L] $, we have $ f(B) \subset E[K_{n}] $, thus $ ( f_{n}(A), K_{1}) \in 2^{E} $ and $ ( f_{n}(B), K_{n}) \in 2^{E} $. Therefore there is an uniform chain transitive from $ A $ to $ B $ in $ 2^{X} $.
\end{proof}
\begin{lemma}
Let $ X $ be a compact Hausdorff space and let $ f: X \longrightarrow X $ be a continuous function. If $ f_{n} $ is uniform chain transitive for some $ n \geq 2 $ then $ f_{n} $ is uniform chain transitive for any $ n \geq 1 $.
\end{lemma}
\begin{proof}
It will therefore suffice to show that given a fixed $ n \geq 2 $, uniform chain transitivity of $ f_{n} $ implies uniform chain transitivity of $ f_{n+1} $. To this end let $ E \in \mathscr{U} $, without loss of generality we may assume that $ A=\left\lbrace a_{1},a_{2}, \ldots , a_{n+1} \right\rbrace  $ and $ B=\left\lbrace b_{1},b_{2}, \ldots , b_{n+1} \right\rbrace  $. Let $ A^{\prime}=\left\lbrace a_{1},a_{2}, \ldots , a_{n} \right\rbrace  $ and $ B^{\prime}=\left\lbrace b_{1},b_{2}, \ldots , b_{n} \right\rbrace  $, $ A^{\prime\prime}=\left\lbrace a_{1},a_{2}, \ldots , a_{n-1} \right\rbrace  $ and $ B^{\prime\prime}=\left\lbrace b_{1},b_{2}, \ldots , b_{n-1} \right\rbrace  $. Since $ f_{n} $ is uniform chain transitive, there are $ E $-chains:\\
$ \{  A^{\prime} = A_{0}, A_{1}, A_{2}, \ldots , A_{r} = A^{\prime\prime} \}  $ of length $ r $ and $ \{B^{\prime\prime}=B_{0},B_{1},B_{2}, \ldots , B_{t}=B^{\prime} \}  $ of length $ t $. Then $ \{ A_{0} \cup \{ a_{n+1} \}, A_{1} \cup \{ f(a_{n+1}) \} , A_{2} \cup \{ f^{2}(a_{n+1}) \}, \ldots ,A_{r} \cup \{ f^{r}(a_{n+1}) \} \}  $ is an $ E $ -chain from $ A $ to $ \{ A^{\prime\prime}\cup \{ f^{r}(a_{n+1}) \} \}  $.\\ Let $ w \in f^{-t}(b_{n+1}) $, then $ \left\lbrace B_{0} \cup\{ w\}, B_{1} \cup \left\lbrace f(w) \right\rbrace , B_{2} \cup \left\lbrace f^{2}(w) \right\rbrace , \ldots ,B_{t} \cup \left\lbrace f^{t}(w) \right\rbrace \right\rbrace $ is an $ E $-chain from $ B^{\prime\prime} \cup \left\lbrace w \right\rbrace  $ to $ B $. Since $ f_{n}(X) $ is uniform chain transitive, then there is an $ E- $chain from $ A^{\prime\prime}\cup \left\lbrace f^{r}(a_{n+1}) \right\rbrace $ to $ B^{\prime\prime} \cup \left\lbrace w \right\rbrace $. Thus, the concatenation this $ E $-chains, we have an $ E $-chain from $ A $ to $ B $ in $ F_{n+1}(X) $. 
\end{proof}
\begin{corollary}
Let $ X $ be acompact Hausdorff space and let $ f: X \longrightarrow X $ be a continuous function. Then the following are equivalent: \\
(i) $ 2^{f} $ is uniform chain transitive,\\
(ii) $ f_{n} $ is uniform chain transitive for some $ n \geq 2 $,\\
(iii) $ f_{n} $ is uniform chain transitive for any $ n \geq 1 $.
\end{corollary}
\begin{lemma}
Let $ X $ be a compact Hausdorff space and let $ f: X \longrightarrow X $ be a continuous function. If  $ f_{n} $ is uniform chain transitive for any $ n \geq 1 $, then $ f $ is uniform chain transitive and for any $ z \in X $ and any entourage $ E $ there are two $ E $-chains from $ z $ to itself with $ co $- prime lengths. 
\end{lemma}
\begin{proof}
We assume in particular, that $ f_{2} $ is uniform chain transitive, let $ z \in X $ and $ E \in \mathscr{U} $ be arbitrary. Since $ f_{2} $ is uniform chain transitive, then there exists an $ E $-chain from $ \left\lbrace z,f(z) \right\rbrace  $ to $ \left\lbrace z \right\rbrace  $ in $ F_{2}(X) $. If this $ E $ -chain has length $ r $, then this implies that there are two different $ E $-chains from $ z $ to $ z $ of length $ r $ and $ r+1 $ as desired. Furthermor, by proposition\ref{A} $ f $ is uniform chain transitive.
\end{proof}
\begin{lemma}
Let $ X $ be a compact Hausdorff space and let $ f: X \longrightarrow X $ be a continuous function.If $ f $ is uniform chain transitive and there exist $ z \in X $ such that for any entourage $ E $ there are two $ E $-chains from $ z $ to itself with $ co $-prime lengths, then $ f_{n} $ is uniform chain transitive for some $ n \geq 2 $.
\end{lemma}
\begin{proof}
Let $ A=\left\lbrace a_{1},a_{2} \right\rbrace  $ and $ B=\left\lbrace b_{1},b_{2} \right\rbrace  $ be two points in $ F_{2}(X) $ and let $ E \in \mathscr{U} $. In order to show uniform chain transitivity of  $ f_{2} $, we will construct an $ E $-chain from $ A $ to $ B $ in $ F_{2}(X) $. For $ i \in \left\lbrace 1,2 \right\rbrace  $, let $ \alpha_{i} $ be the shortest $ E $-chain from $ a_{i} $ to $ z $, let $ \beta_{i} $ be the shortest $ E $-chain from $ z $ to $ b_{i} $, and let $ \gamma_{i} $ be an $ E $-chain from $ z $ to $ z $ of length $ p_{i} $. Also, let us assume that the length of $ \alpha_{i} $ is $ k_{i} $ and length of $ \beta_{i} $ is $ m_{i} $. In order to have an $ E $-chain from $ A $ to $ B $ in $ F_{2}(X) $, we need to find positive number $ r $ and $ t $ satisfying: $ k_{1} + r \cdot p_{1} + m_{1} = k_{2} + t \cdot p_{2} + m_{2} $.
Without loss of generality, we may assume that $ k_{1} + m_{1} > k_{2} + m_{2} $. Thus, the numbers $ r $ and $ t $ should satisfy: $ k_{1} + m_{1} - \left( k_{2} + m_{2} \right) = t \cdot p_{2} - r \cdot p_{1} $. Since $ (p_{1} , p_{2})=1 $, then it is always possible to find the numbers $ r $ and $ t $.
\end{proof}
\begin{lemma} \label{B}
Let $ X $ be a compact Hausdorff space, let $ Y $ be a uniform space and let $ f: X \longrightarrow X $, and $ g: Y \longrightarrow Y $ be continuous such that there exists $ h: X \longrightarrow Y $ onto and continuous such that $ h \circ f = g \circ h $. If $ f $ is uniform chain transitive, then $ g $ is uniform chain transitive.
\end{lemma}
\begin{proof}
Let us consider $ H \equiv h \times h:X \times X \longrightarrow Y \times Y $ given by $ H(x,y) = \left(  h(x) , h(y) \right)  $. For an entourage $ E $ of $ Y $, observe that $ D = H^{-1}(E) $ is an entourage of $ X $. Let $ y_{1},y_{2} \in Y $ be disjoint points. Since $ H $ is onto, there are $ x_{1} $ and $ x_{2} $ in $ X $ such that $ h(x_{1}) = y_{1} $ and $ h(x_{2}) = y_{2} $. Also, since $ X $ is compact, there is entourage $ D \in \mathscr{U} $ such that if $ (a,b) \in D $, then $ \left( h(a) , h(b) \right) \in E $, for all $ a,b \in X $. Hence, there is a $ D $-chain in $ X $, $ \left\lbrace x_{1}=z_{0},z_{1}, \ldots ,z_{r}=x_{2} \right\rbrace  $ with $ r \geq 1 $. Thus, for $ i \in \left\lbrace 0,1, \ldots , r-1 \right\rbrace  $, we have that $ \left( f(z_{i}) , z_{i+1} \right) \in D $, then $ \left( g \left( h \left( z_{i} \right) \right) , h \left( z_{i+1} \right) \right)  = \left( h \left( f\left( z_{i} \right) \right) , h \left( z_{i+1} \right) \right) \in E $. Therefore, the sequence $ \{ y_{1} = h(z_{0}), h(z_{1}), \ldots , h(z_{r}) = y_{2} \}  $ is an $ E $-chain in $ Y $.
\end{proof}
\begin{proposition}
Let $ X $ be a compact Hausdorff space, let $ f: X \longrightarrow X $ be a continuous function and let $ n\geq 1 $.
The following are equivalent:\\
(1) $ f^{(n)} : X^{(n)} \longrightarrow X^{(n)} $ is uniform chain transitive,\\
(2) $ f_{n} : F_{n}(X) \longrightarrow F_{n}(X) $ is uniform chain transitive.
\end{proposition}
\begin{proof}
Let $ h: X^{(n)} \longrightarrow F_{n}(X) $ given by for each $ ( x_{1} , x_{2}, \ldots , x_{n} )  $ in $ X^{n} $, 
\begin{align*}
h\left( \left( x_{1} , x_{2} , \ldots , x_{n} \right) \right) = \left\lbrace x \in X \ \ : x=x_{i} ,\ \ for \ \ some\ \ i \in \left\lbrace 1,2, \ldots , n \right\rbrace \right\rbrace .
\end{align*}
Also, if $ ( x_{1} , x_{2} , \ldots , x_{n} ) \in X^{(n)} $, then 
\begin{align*}
\left(  h \circ f^{(n)}\right) \left( \left( x_{1} , x_{2} , \ldots , x_{n} \right) \right) &= h \left( \left( f (x_{1}) , f ( x_{2}) , \ldots , f (x_{n}) \right) \right) \\&= \left\lbrace  x \in X \ \ : \ x=f(x_{i}), \ \ for \ \ some \ \ i \in \left\lbrace 1 , 2 , \ldots , n \right\rbrace  \right\rbrace \\&= f_{n} \left(  \left\lbrace  x \in X \ \ : x = x_{i}, \ \  for \ \ some \ \ i \in \left\lbrace 1 , 2 , \ldots , n \right\rbrace \right\rbrace \right)  \\&= f_{n} \left( h \left( \left( x_{1} , x_{2} , \ldots , x_{n} \right) \right) \right) \\&= f_{n} \circ h \left( \left(  x_{1} , x_{2} , \ldots , x_{n} \right) \right) .
\end{align*}
Hence, by lemma \ref{B}, we have that (1) implies (2).\\
To see that (2) implies (1), let $ ( x_{1} , x_{2} , \ldots , x_{n} ) , ( y_{1} , y_{2} , \ldots , y_{n} ) \in X^{(n)} $ and let $ E \in \mathscr{U} $. Let $ z \in X $, then $ h \left( \left( x_{1} , x_{2} , \ldots , x_{n} \right) \right) , h \left( \left(  y_{1} , y_{2} , \ldots , y_{n} \right) \right)  $ and $ \{z\} $ are in $ F_{n}(X) $. By hypothesis there are $ E $-chains in $ F_{n}(X) $ in the following way: \\
$ \left\lbrace  h \left( \left(  x_{1} , x_{2} , \ldots , x_{n} \right) \right)  = A_{0} , A_{1} , \ldots , A_{m_{1}} = \left\lbrace z \right\rbrace  \right\rbrace $ and $ \left\lbrace \left\lbrace z \right\rbrace  = B_{0} , B_{1} , \ldots , B_{m_{2}}=h \left( \left( y_{1} , y_{2} ,\ldots , y_{n} \right) \right) \right\rbrace  $. For each $ i \in \left\lbrace  1 , 2 , \ldots , n \right\rbrace  $, we have induced $ E $-chain in $ X $, $ \left\lbrace x_{i} = a_{0}^{i} , a_{1}^{i} , \ldots , a_{m_{1}}^{i} = z \right\rbrace  $ and $ \left\lbrace z = b_{0}^{i} , b_{1}^{i} , \ldots , b_{m_{2}}^{i} = y_{i} \right\rbrace  $, where $ a_{j}^{i} \in A_{j} $ and $ b_{t}^{i} \in B_{t} $, for each $ j \in \left\lbrace 0 , 1 , \ldots , m_{1} \right\rbrace  $ and each $ t \in \left\lbrace 0 , 1 , \ldots , m_{2} \right\rbrace  $. Thus, the sequence $ \left\lbrace D_{0} , D_{1} , \ldots , D_{m_{1}} , D_{m_{1}+1} , \ldots , D_{m_{1}+m_{2}} \right\rbrace  $, given by:
\[
 \rm   D_{i}=\left \{ 
 \begin{array}{cc}
 ( a_{i}^{1}, a_{i}^{2}, \ldots , a_{i}^{n} )&  if ~  i \in \{ 0, 1, \ldots , m_{1} \} \\
 ( b_{i-m_{1}}^{1}, b_{i-m_{1}}^{2}, \ldots , b_{i-m_{1}}^{n} )   &   if ~  i \in \{ m_{1}+1, m_{1}+2, \ldots , m_{1}+m_{2} \},
 \end{array} \right. \]
is an $ E $-chain in $ X^{(n)} $ from $ ( x_{1} , x_{2} , \ldots , x_{n} ) $ to $ ( y_{1} , y_{2} , \ldots , y_{n} ) $.
\end{proof}
\begin{lemma}
Let $ X $ be a compact Hausdorff space, let $ f: X \longrightarrow X $ be a continuous function. If $ f^{(n)} $ is uniform chain transitive for any $ n \geq 1 $, then for every $ z \in X $ and each entourage $ E $, there is a positive integer $ n \geq 1 $ such that for each $ x \in X \setminus \{ z \} $, there is an $ E- $ chain of length $ n $ from $ z $ to $ x $.
\end{lemma}
\begin{proof}
Let $ W \subseteq X $ be a non-empty open set and let $ E \in \mathscr{U} $ and let $ D \in \mathscr{U} $ such that $ D^{2} \subset E $. Since $ X $ is compact, there are $ x_{1} , x_{2} , \ldots , x_{k} \in X $ such that $ X = \left\lbrace Int_{X} \left( D[x_{i}] \right) \right\rbrace _{i=1}^{k} $. By hypothesis, $ f^{(k)} : X^{(k)} \longrightarrow X^{(k)} $ is uniform chain transitive. Let $ w \in W $, since $ ( w , w , \ldots , w ) , ( x_{1} , x_{2} , \ldots , x_{k} ) \in X^{(k)} $, there is an $ D $-chain $ \left\lbrace \left( w, w, \ldots , w \right) = \left( a_{1}^{0}, a_{2}^{0}, \ldots , a_{k}^{0} \right) , \left( a_{1}^{1}, a_{2}^{1}, \ldots , a_{k}^{1} \right) , \ldots , \left( a_{1}^{r} , a_{2}^{r} , \ldots , a_{k}^{r} \right) = \left( x_{1}, x_{2}, \ldots , x_{k} \right)  \right\rbrace  $ \\ let $ x \in X $, then $ x \in Int_{X} \left( D[ x_{j} ] \right) $ for some $ j \in \left\lbrace 1, 2, \ldots , k \right\rbrace  $. Hence,\\ $ \left( f \left( a_{j}^{r-1} \right) , x \right) \subset \left( f \left( a_{j}^{r-1} \right) , x_{j} \right) \circ \left( x_{j} , x \right) \in D \circ D = D^{2} \subset E $.Therefore, the sequence: \\ 
$ \left\lbrace w = a_{j}^{0}, a_{j}^{1}, \ldots , a_{j}^{r-1}, x \right\rbrace  $ is an $ E $-chain from $ w $ to $ x $ of length $ k $. 
\end{proof} 
\begin{lemma} 
 Let $ X $ be a compact Hausdorff space, let $ f: X \longrightarrow X $ be a continuous function.  If for every $ z \in X $ and each entourage $ E $, there is a positive integer $ n \geq 1 $ such that for each $ x \in X \setminus \{ z \} $, there is an $ E $-chain of length $ n $ from $ z $ to $ x $ then $ f $ is exact by uniform chains.
\end{lemma}
\begin{proof}
Let $ E \in \mathscr{U} $ and let $ W $ be a non-empty open subset f $ X $. Since $ W $ is non-empty, let $ w \in W $. Hence, there is a positive integer $ n \geq 1 $ such that for each $ x \in X\setminus  \left\lbrace w \right\rbrace $, there is an $ E $-chain of length $ n $ from $ w $ to $ x $ as desired.
\end{proof}
\begin{lemma}
Let $ X $ be a compact Hausdorff space, let $ f: X \longrightarrow X $ be a continuous function. If $ f $ is exact by uniform chains, then $ f $ is uniform chain weakly mixing. 
\end{lemma}
\begin{proof}
Suppose that $ f $ is exact by uniform chains. To see that $ f^{(2)} : X^{(2)} \longrightarrow X^{(2)} $ is uniform chain transitive, let $ \left( x_{1} , x_{2} \right) , \left( y_{1} , y_{2} \right)  \in X^{(2)} $ and let $ E \in \mathscr{U} $. Let $ W_{1} = Int_{X} \left( E[ f (x_{1})] \right)  $ and $ W_{2} = Int_{X} \left( E[ f ( x_{2})] \right)  $. By hypothesis, there are positive integers $ m_{1} , m_{2} $ such that for each $ x \in X $, there exist $ w_{1} \in W_{1} , w_{2} \in W_{2} $, and $ E $-chains, $ \left\lbrace w_{1} = a_{0}, a_{1}, \ldots , a_{m_{1}} = x \right\rbrace  $ and $ \left\lbrace w_{2} = b_{0}, b_{1}, \ldots , b_{m_{2}} = x \right\rbrace  $. Hence, for $ f(x_{2}) $, there are $ w_{1}^{\prime} \in W_{1} $ and an $ E $-chain $ \left\lbrace w_{1}^{\prime} = c_{0}, c_{1}, \ldots , c_{m_{1}} = f (x_{2}) \right\rbrace $. Similarly, for $ f(x_{1}) $, there are $ w_{2}^{\prime} \in W_{2} $ and an $ E $-chain $ \left\lbrace w_{2}^{\prime} = d_{0}, d_{1}, \ldots , d_{m_{2}} = f(x_{1}) \right\rbrace  $. Then, for $ y_{1} $, we have $ w_{2}^{\prime\prime} \in W_{2} $ and an $ E $-chain $ \left\lbrace w_{2}^{\prime\prime} = s_{0}, s_{1}, \ldots , s_{m_{2}} = y_{1} \right\rbrace  $. Analogously, for $ y_{2} $, there are $ w_{1}^{\prime\prime} \in W_{1} $ and an $ E $-chain $ \left\lbrace w_{1}^{\prime\prime} = t_{0}, t_{1}, \ldots , t_{m_{1}} = y_{2} \right\rbrace  $. Thus, the sequences:\\
$ \left\lbrace x_{1} , w_{1}^{\prime} = c_{0}, c_{1}, \ldots , c_{m_{1}} = f (x_{2}), w_{2}^{\prime\prime} = s_{0}, s_{1}, \ldots , s_{m_{2}} = y_{1} \right\rbrace  $ \\ and\\ $ \left\lbrace x_{2},  w_{2}^{\prime} = d_{0}, d_{1}, \ldots , d_{m_{2}} = f(x_{1}), w_{1}^{\prime\prime} = t_{0}, t_{1}, \ldots , t_{m_{1}} = y_{2} \right\rbrace $ \\ are $ E $-chains of the same length. Hence, the sequence:\\ 
$ \left\lbrace ( x_{1} , x_{2} ) , ( w_{1}^{\prime} , w_{2}^{\prime} ) = ( c_{0} , d_{0} ), \ldots , ( s_{m_{2}} , t_{m_{1}} ) = (y_{1} , y_{2}) \right\rbrace  $ is an $ E $-chain in $ X^{(2)} $.
\end{proof} 
\begin{lemma} 
Let $ X $ be a compact Hausdorff space, let $ f: X \longrightarrow X $ be a continuous function. If $ f $ is uniform chain weakly mixing, then $ f $ is exact by uniform chains.????
\end{lemma}
\begin{lemma}
Let $( X,\mathscr{U}) $ be a compact Hausdorff space and let $ f: X \longrightarrow X $ be a continuous function. If $ f^{(n)} $ is uniform chain transitive for any $ n \geq 1 $, then $ f $ is totally uniform chain transitive.
\end{lemma}
\begin{proof}
Let $ n \geq 1 $ and $ E \in \mathscr{U} $. Let $ D_{i+1} \in \mathscr{U} $ and $ D_{i}^{\prime} \in \mathscr{U} $ be such that $ D_{i}^{\prime 2} \subset D_{i+1} $ for $ i \in \left\lbrace 1, 2, \ldots , n-1 \right\rbrace  $. Since $ X $ is compact, there is a sequence $ D_{1} \subset D_{2} \subset \ldots \subset D_{n} = E $ such that if $ ( a, b) \in D_{i} $, then $ \left( f(a), f(b) \right) \in D_{i}^{\prime} $, for each $ i \in \left\lbrace 1, 2, \ldots , n-1 \right\rbrace  $. Let $ D^{2} = D_{1} $. By hypothesis, there is an uniform $ D $-chain in $ X^{(n)} $, $ \left\lbrace \left( x, y, y, \ldots , y \right) = \left( z_{1}^{0}, z_{2}^{0}, \ldots , z_{n}^{0} \right) , \left( z_{1}^{1}, z_{2}^{1}, \ldots , z_{n}^{1} \right) , \ldots , \left( z_{1}^{t}, z_{2}^{t}, \ldots , z_{n}^{t} \right) = \left( y, y, \ldots , y \right) \right\rbrace  $ \\ with $ t \geq 1 $. Thus, the sequence: \\ 
$ \left\lbrace x = z_{1}^{0}, z_{1}^{1}, z_{1}^{2}, \ldots , z_{1}^{t} = y = z_{2}^{0}, z_{2}^{0}, \ldots , z_{2}^{t} = y = z_{3}^{0}, z_{3}^{1}, \ldots , z_{n}^{t} = y \right\rbrace  $ is an uniform $ D $-chain of length $ tn $. Renaming the elements of this $ D $-chain, we can write it as: $ \left\lbrace x = a_{0}, a_{1}, \ldots , a_{tn} = y \right\rbrace  $. Since $ \left( f(x), a_{1} \right) \in D \subset D_{1} $, then $ \left( f^{2}(x), f(a_{1}) \right) \in D_{1}^{\prime} $. Also, we have that $ \left( f(a_{1}), a_{2} \right) \in D_{1}^{\prime} $, then \\ $ \left( f^{2}(x), a_{2} \right) = \left( f^{2}(x), f(a_{1}) \right) \circ \left( f(a_{1}), a_{2} \right) \in D_{1}^{\prime} \circ D_{1}^{\prime} = D_{1}^{\prime 2} \subset D_{2} $. \\ Again, since $ \left( f^{2}(x), a_{2} \right) \in D_{2} $ then $ \left(  f^{3}(x), f( a_{2} ) \right) \in D_{2}^{\prime} $. But $ \left( f(a_{2}), a_{3} \right) \in D_{2}^{\prime} $ then $ \left( f^{3}(x) , a_{3} \right) \in D_{3} $. Continue this proces, we finally get $ \left( f^{n}(x) , a_{n} \right) \in D_{n} = E $. Analogously, $ \left( f^{n}(a_{in}) , a_{(i+1)n} \right) \in E $, for each $ i \in \left\lbrace 0, 1, \ldots , t-1 \right\rbrace  $. Thus, the sequence: \\
$ \left\lbrace x = z_{0}, z_{1} = a_{n}, z_{2} = a_{2n}, \ldots , z_{t} = a_{tn} = y\right\rbrace  $ is an uniform $ E $-chain in $ X $ for the function $ f^{n} $. Hence, $ f^{n} $ is uniform chain transitive for every $ n \geq 1 $.
\end{proof}
\begin{lemma}
Let $( X,\mathscr{U}) $ be a compact Hausdorff space and let $ f: X \longrightarrow X $ be a continuous function. If  $ f $ is totally uniform chain transitive, then $ f $ is uniform chain weakly mixing.
\end{lemma}
\begin{proof}
Let $ ( x_{1}, x_{2} ) , ( y_{1}, y_{2} ) \in X^{(2)} $, let $ E \in \mathscr{U} $ and let $ w \in X $. Since $ f $ is totally uniform chain transitive, in particular $ f $ is uniform chain transitive. Then there is an $ E $-chain $ \alpha_{1} $, from $ x_{1} $ to $ w $ of length $ m_{1} $ and also there is an $ E $-chain $ \alpha_{2} $, from $ w $ to $ x_{1} $ of length $ m_{2} $. Hence, the function $ f^{m_{1}+m_{2}} $ from $ x_{2} $ to $ x_{1} $ of length $ m_{3} $, $ \alpha_{3} = \left\lbrace x_{2} = a_{0}, a_{1}, \ldots , a_{m_{3}} = x_{1} \right\rbrace  $. Since the sequence $ \alpha_{3}^{\prime}=\{ x_{2} = a_{0}, f(a_{0}), f^{2}(a_{0}), \ldots , f^{m_{1}+m_{2}-1}(a_{0}), a_{1}, f(a_{1}), \ldots , f^{m_{1}+m_{2}-1}(a_{1}), a_{2}, \ldots , a_{m_{3}-1}, f(a_{m_{3}-1}), \ldots,$ $ f_{m_{1}+m_{2}-1}(a_{m_{3}-1}), a_{m_{3}} = x_{1} \} $ is an $ E $-chain for $ f $ from $ x_{2} $ to $ x_{1} $ of length $ m_{3}(m_{1} + m_{2}) $. Therefore, the sequence $ \alpha_{1} + m_{3}(\alpha_{2} + \alpha_{1}) $ and $ \alpha_{3}^{\prime} + \alpha_{1} $ are $ E $-chains from $ x_{1} $ to $ w $ and from $ x_{2} $ to $ w $, respectively, of length $ m_{3}( m_{1} + m_{2} ) + m_{1} $. Hence, these $ E $-chains of the same length induce an $ E $-chain $ \alpha $ in $ X^{(2)} $ from $ ( x_{1} , x_{2} ) $ to $ ( w,w) $. Analogously, we can construct an $ E $-chain $ \beta $ in $ X^{(2)} $ from $ ( w,w ) $ to $ ( y_{1}, y_{2} ) $. Thus, the concatenation $ \alpha + \beta $ is an $ E $-chain in $ X^{(2)} $ from $ ( x_{1}, x_{2} ) $ to $ ( y_{1}, y_{2} ) $. Hence, $ f^{(2)} : X^{(2)} \longrightarrow X^{(2)} $ is uniform chain transitive.
\end{proof}

Note that if $ f $ is uniform chain weakly mixing, then $ f^{(2)}: X^{(2)} \rightarrow X^{(2)} $ is uniform chain transitive.\\
\begin{lemma}
Let $ X $ be a non-empty compact connected uniform space, and let $ f: X \longrightarrow X $ be a continuous function. Suppose that $ f $ is uniform chain transitive. If $ V $ and $ W $ are two non-empty proper disjoint open subsets of $ X $, then either $ Cl_{X} \left(  f(V) \right) \nsubseteq W $ or $ Cl_{X} \left( f(W) \right) \nsubseteq V $. 
\end{lemma}
\begin{proof}
Assume that the lemma is not true. Let $ V , W $ be non-empty proper disjoint open subsets of $ X $ such that $ Cl_{X} \left( f(V) \right) \subset W $ and $ Cl_{X} \left( f(W) \right) \subset V $. Let $ E_{1} , E_{2} \in \mathscr{U} $ such that $ Int_{X} \left( E_{1} [ Cl_{X} \left(  f(V) \right)] \right)  \subset W $ and $ Int_{X} \left( E_{2} [ Cl_{X} \left( f(W) \right)] \right)  \subset V $. Since $ X $ is connected, there is $ z \in X\setminus (V \cup W) $. If $ E = \bigcap \left\lbrace E_{1}, E_{2} \right\rbrace  $, then for $ x \in V $, $ f(x) \in Cl_{X} \left( f(V) \right)  $. Thus, if $ x_{1} \in X $ is such that $ \left( x_{1}, f(x) \right)  \in E $, then $ x_{2} \in V $. Hence, there is no an $ E $-chain from $ x $ to $ z $, which is cotradiction. 
\end{proof}
\begin{corollary}
Let $ X $ be a non-empty compact connected space and $ f: X \longrightarrow X $ a continuous function. Then $ f $ is uniform chain transitive ifand only if one, and hence all of the conditions in last lemmas hold. Furthermore all of these condition are equivalent to requiring that $ f $ is uniform chain recurrent.\\ 
Using the fact that if $ X $ is a non-empty compact connected space, then $ 2^{X} $ is also a non-empty compact connected space.
\end{corollary}
\begin{corollary}
Let $ X $ be non-empty compact connected space and let $ f: X \longrightarrow X $ be a continuous function. The following are equivalent: \\
(1) $ f $ is uniform chain transitive;\\
(2) $ 2^{f} $ is uniform chain transitive;\\
(3) $ 2^{f} $ is totally uniform chain transitive;\\
(4) $ 2^{f} $ is uniform chain weakly mixing;\\
(5) $ 2^{f} $ exact by uniform chains;\\
(6) $ 2^{f} $ is uniform chain recurrent.
\end{corollary}

\begin{theorem} 
Let $( X, \mathscr{U} ) $ be uniform space and let $ f: X \longrightarrow X $ be a continuous function such that $ f_{2} $ is uniform chain transitive. If $ X $ has finitely many components, then $ X $ is connected.
\end{theorem}
\begin{proof}
Let $ W $ be open subset of $ X $ containing $ z $. Then there exists $ E \in \mathscr{U} $ such that $ E[z] \subset W $. Assume that $ X $ has $ r $ components say $ W_{1}, W_{2}, \ldots , W_{r} $. Since $ X $ is compact, there exist $ z_{1}, z_{2}, \ldots , z_{r} $ in $ X $ such that $ X = \bigcup_{i=1}^{r} W_{z_{j}} $. Let $ W_{1} = W_{Z_{1}} $ and for each $ j \in \left\lbrace 2, \ldots , r \right\rbrace  $, let $ W_{j} = W_{z_{j}} \setminus \bigcup_{k=1}^{r-1} W_{z_{k}} $. Then $ \mathcal{W} = \left\lbrace W_{1}, W_{2}, \ldots , W_{r} \right\rbrace  $ is a finite open cover of $ X $ such that $ W_{j} \cap W_{k} = \emptyset  $ if $ j \neq k  $. Hence there exists $ D \in \mathscr{U} $ such that $ \mathfrak{C} (D) $ refines $ \mathcal{W} $. If $ A = \left\lbrace a_{1}, a_{2} \right\rbrace \subseteq W_{1} $ and $ B = \left\lbrace b_{1}, b_{2} \right\rbrace  $ such that $ b_{i} \in W_{i}, i \in \left\lbrace 1,2 \right\rbrace  $, then there is no an $ D $-chain from $ A $ to $ B $ in $ F_{2}(X) $, which contradicts our assumption. 
\end{proof}
\bibliographystyle{acm}

\begin{thebibliography}{33}

\bibitem{MR1687407}
James, Ioan, Topologies and uniformities, Springer, New York, 1987.

\bibitem{MR3528853} Erfanian Omidvar, M. and Ahmadi, S. A. and Darban Maghami, N., Topological entropy and topological pressure of a
              homeomorphism on a dynamical space, Politehn. Univ. Bucharest Sci. Bull. Ser. A Appl. Math. Phys.
 78 (2016) 31-42.

\bibitem{ahmadi}
Ahmadi, S.A., Shadowing, ergodic shadowing and uniform spaces,
Filomat 31 (2017) 5117-5124.


\bibitem{WU2019145}
"Xinxing Wu , Yang Luo , Xin Ma , Tianxiu Lu, Rigidity and sensitivity on uniform spaces,
Topology and its Applications 252 (2019) 145 - 157.

\end{thebibliography}


\end{document}